\documentclass[leqno,10pt]{amsart}

\usepackage{amsmath}

\usepackage{amssymb}
\usepackage{graphicx}

\usepackage{color}
\usepackage[dvipsnames]{xcolor}

\newcommand{\alertm}[1]{%
  \marginpar{%
    \ifodd\value{page} \raggedright \else \raggedleft \fi
    \footnotesize{\textcolor{Green}{#1}}
  }
}

\newtheorem{thm}{Theorem}

\newtheorem{prop}[thm]{Proposition}

\theoremstyle{definition}
\newtheorem{defn}{Definition}

\theoremstyle{remark}
\newtheorem{remark}{Remark}

\newtheorem{assumption}{Assumption}

\numberwithin{equation}{section}

\def\R{\mathbb{R}}

\def\Z{\mathbb{Z}}

\def\F{\mathcal{F}}

\def\a{\mathfrak{a}}

\def\m{\mathfrak{m}}

\def\I{\mathrm{I}_{\s}}

\def\Q{\mathcal{Q}}

\def\RW{\mathcal{M}}
\def\RWC{\widetilde{\RW}}
\def\S{\mathbb{S}^3}
\def\J{\mathrm{J}_{\s}}
\def\s{\mathfrak{s}}
\def\p{\mathfrak{p}}
\def\m{\mathfrak{m}}
\def\cm{\widetilde{\m}}
\def\cp{\widetilde{\p}}
\def\a{\mathfrak{a}}

\def\H{\mathcal{H}}
\def\Lr{\ell_{\s}}
\def\LrC{\ell_{\S}}
\def\LrE{\ell_{1}}
\def\F{\phi_{\s}}
\def\E{\mathcal{E}}
\def\SS{\mathbb{S}^3}
\def\S{\mathrm{S}}
\def\Q{\mathrm{Q}}

\begin{document}

\title[]{Critical Robertson--Walker universes}

\author{Olimjon Eshkobilov}
\address{(O. Eshkobilov) Dipartimento di Matematica ``Giuseppe Peano'', Universit\`a di Torino,
Via Carlo Alberto 10, I-10123 Torino, Italy}
\email{olimjon.eshkobilov@edu.unito.it, olimjon.eshkobilov@polito.it
}

\author{Emilio Musso}
\address{(E. Musso) Dipartimento di Scienze Matematiche, Politecnico di Torino,
Corso Duca degli Abruz\-zi 24, I-10129 Torino, Italy}
\email{emilio.musso@polito.it}

\author{Lorenzo Nicolodi}
\address{(L. Nicolodi) Di\-par\-ti\-men\-to di Scienze Ma\-te\-ma\-ti\-che, Fisiche e Informatiche,
Uni\-ver\-si\-t\`a di Parma, Parco Area delle Scienze 53/A,
I-43124 Parma, Italy}
\email{lorenzo.nicolodi@unipr.it}

\thanks{Authors partially supported by
PRIN 2015-2018 ``Variet\`a reali e complesse: geometria, to\-po\-lo\-gia e analisi ar\-mo\-ni\-ca'';
by the GNSAGA of INDAM; and by the FFABR Grant 2017 of MIUR.
The present research was also partially supported by MIUR grant
``Dipartimenti di Eccellenza'' 2018–2022, CUP: E11G18000350001, DISMA, Politecnico
di Torino.}

\subjclass[2010]{53C50, 53B30, 53Z05, 83C15, 33E05, 58E30}



\keywords{Lorentz metrics; Robertson--Walker spacetimes; perfect fluid;
conformal time; cyclic models in general relativity; weak, dominant, and strong energy conditions;
elliptic functions and integrals; Weierstrass $\wp$-functions}

\begin{abstract}

The integral of the energy density function $\m$ of
a closed Robertson--Walker (RW) spacetime
with source a perfect fluid
and cosmological
constant $\Lambda$
gives rise to an action functional on the space of scale functions of RW spacetime metrics.
%
%
This paper studies closed RW spacetimes which are critical for this functional,
subject to
volume-preserving variations ($\m$-critical RW spacetimes).
A complete classification of $\m$-critical RW spacetimes is given
and explicit solutions in terms of Weierstrass elliptic functions and their degenerate forms are computed.
The standard energy conditions
(weak, dominant, and strong)
as well as the cyclic property
of $\m$-critical RW spacetimes
are discussed.
\end{abstract}

\maketitle

\section{Introduction}

Robertson--Walker (RW) spacetimes
with source a perfect fluid
are the most basic cosmological models in general relativity.
Despite their old 
history \cite{Fri,Le,Rb,Wa},
RW cosmological models still provide a valuable testing ground for
new ideas and theories. These include
gravitational thermodynamics \cite{Ca,TB,ZRL}, relativistic diffusion \cite{AN},
cyclic and conformal cyclic cosmology \cite{BC,Pn1,Pn3,Tod,Todcourse},
Lorentz conformal geometry \cite{AL,BCDG,Fr}, super-symmetry and string theory \cite{BMMO,RO}.

\vskip0.1cm

In this work we will investigate a natural variational problem for closed RW
spacetimes filled with a perfect fluid and satisfying the Einstein equation with a cosmological constant.
More precisely, let $\mathcal M(\s) = ( \mathrm I_\s \times\mathbb  S^3, \ell_\s = -dt^2 + \s^2(t) g)$ be
a closed RW spacetime,
where $\s$, the scale function, is a nonnegative smooth function on an open
interval $\mathrm I_\s \subset \R$ and $g$ is the standard metric
of the 3-sphere $\mathbb S^3$.
A closed RW cosmological model (universe) is a closed RW spacetime $\mathcal M(\s)$
whose RW metric $\ell_\s$ satisfies the Einstein equation with cosmological constant $\Lambda$
and source given by
the stress-energy tensor of a perfect fluid, $\mathrm T =\m dt^2+\p \s^2 g$,
where the energy density $\m$ and the pressure $\p$ are scalar functions of $t$ only.
%
%
%
Accordingly,
the integral
\begin{equation}\label{action}
   \mathfrak{M}[\s;\Lambda]=\int_{\mathcal{D}}\m \,dV_{\ell_\s}
     \end{equation}
defines an action functional on the space of scale functions of RW metrics.
Here $dV_{\ell_\s}$ stands for the volume element of the RW metric $\ell_\s$ and
$\mathcal{D} = [t_0,t_1]\times \mathbb S^3$
is a compact domain of $\mathrm I_\s \times\mathbb  S^3$.
The critical points of \eqref{action} with respect to volume-preserving variations of the
RW metric
are called {\em $\m$-critical RW universes} (or {\em critical models}, for short).
%
%
After deriving
the variational equation,
which is a 2nd order ODE in
the scale function $\s$, and computing the equation of state fulfilled by
the energy density $\m$ and the pressure $\p$,
%
%
we rewrite the variational equation in terms of the conformal time
$\tau$, defined by $d\tau = dt/\mathfrak s$.
The advantage of this coordinate transformation is that
the resulting {\em conformal scale function} can be described as a real form
of a Weierstrass $\wp$-function.
According to whether the discriminant of the cubic polynomial of such an associated Weierstrass
$\wp$-function is different or equal to zero, a critical model is said to be of {\it general}
or {\it exceptional} type.
%
%
The critical models of exceptional type can be easily described in terms of trigonometric or hyperbolic functions
(cf. Section \ref{s:1}).
The main purpose of this paper is to study the critical models of general type.
We will describe their structure and discuss their energy conditions and periodicity properties.

\vskip0.1cm
Special and elliptic functions have been already used in the literature to describe particular classes
of RW universes \cite{BC,BYYM,DA,DW}. However, this was done by imposing the equation of state a priori,
while our models are obtained from a rather natural variational principle. Weierstrass elliptic functions
are also used
in the study of the restricted evolution problem
within
the more general class of
Lema\^{\i}tre--Tolman--Bondi spacetimes \cite{BCL}. In addition, the theory of elliptic functions plays a relevant
role in the study of null geodesics in Schwarzschild spacetime \cite{GV},
closed conformal geodesics in Euclidean space
\cite{MNcag1},
higher-order variational problems
for null curves in Lorentz space forms \cite{FGL,GM,KP,MNcqg,MNsicon,MNnonlinearity,NFS2,NMMK,NR},
conformally invariant
variational problems for timelike curves in the Einstein static universe \cite{DMNna},
and other geometric variational problems \cite{JMN-JPA,MMN,MNjmp,MNforum}.

\vskip0.1cm
The paper is organized as follows.
Section \ref{s:1} collects some basic facts about
closed RW cosmological models \cite{HE,ON,Wald,We}
and derives the variational equation and the equation of state satisfied by the
scale function of a critical model (cf. Proposition \ref{Proposition1}). As a byproduct,
it follows that the critical scale functions depend
on two internal parameters, denoted by $\a$ and $\H$.
Proposition \ref{Proposition2} writes the
conformal scale function, the conformal energy density, and the conformal pressure,\footnote{i.e., the
scale function,
the energy density, and the pressure written in terms of the conformal time.} in terms of real forms
of Weierstrass $\wp$-functions.
A critical model is said to be of {\it general}
or {\it exceptional} type, according to whether the discriminant of the cubic polynomial of the associated
Weierstrass $\wp$-function is different or equal to zero.
The general critical models are further divided in two classes, namely Class I and Class II,
depending on the positivity or negativity of the discriminant.
The section terminates with the discussion and
the explicit description of the exceptional critical models.

\vskip0.1cm
Section \ref{s:2} studies the critical models of Class I. Theorem \ref{Teorema1} explicitly describes
the critical models of Class I, with $\H<0$, and determines the bounds
satisfied by
the cosmological constant
in order that the model fulfils the weak, dominant, or strong energy conditions.
It is also shown that these solutions are cyclic and nondegenerate. Theorem \ref{Teorema2} provides similar
results for the critical models of Class I, with $\H>0$.
In the latter case,
the scale function vanishes along the disjoint union of a countable family of totally umbilical, connected,
spacelike hypersurfaces.

\vskip0.1cm
Section \ref{s:3} is devoted to the analysis of critical models of Class II. We distinguish
two possible types: negative or positive. Theorem \ref{Teorema3} deals with the critical models
of Class II and negative type. We analyze the bounds on the cosmological constant so that the
weak, the dominant, or the strong energy conditions are satisfied. We also show that these models are
cyclic and degenerate along countably many totally umbilical, connected, spacelike hypersurfaces.
In Theorem \ref{Teorema4}, we consider critical models of Class II and positive type, with $\H>0$.
These models are not cyclic and their Lorentz quadratic forms are degenerate along a totally umbilical,
connected,
spacelike hypersurface. The conformal scale function tends to infinity in finite time.
Finally, in Theorem \ref{Teorema5} we prove similar results for critical models of Class II and
positive type, but with $\H<0$.

\vskip0.1cm
Summarizing, the only genuine cyclic $\m$-critical models are those corresponding to
Class I, with $\H<0$. In all other cases, the scale function either vanishes along a totally
umbilical spacelike hypersurface, or tends to $+\infty$ in finite time.
\vskip0.1cm

As a basic reference for Weierstrass elliptic functions we use \cite{La} .
We acknowledge the use of the software {\sc Mathematica} for
symbolic and numerical computations throughout the paper.

\section{Preliminaries}\label{s:1}

In this section, after recalling the basic definitions, we derive the variational equation for the
scale function of an $\m$-critical RW universe and write the conformal scale function in terms of
a real form of a Weierstrass $\wp$-function.
We find the equation of state and classify the
$\m$-critical RW universes into three classes.

\subsection{Closed RW universes and their conformal models}

Let $g$ be the standard Riemannian metric of the unit 3-sphere $\SS\subset \R^4$. Consider a closed
RW {\it spacetime} $\RW:=\I\times_{\s} \SS$  with a {\it nonnegative scale function} $\s:\I\subset \R\to \R^+$, endowed with the (possibly degenerate) Lorentz metric $\Lr=-dt^2+\s^2(t)\,g$, the {\it cosmological constant} $\Lambda$ and a {\it perfect fluid stress-energy tensor}
\[
  \mathrm{T}=\m dt^2+\p \s^2 g,
  \]
where $\m :\I\to \R^+\cup \{+\infty\}$ is the {\it energy density function} and $\p:\I\to \R\cup \{\pm \infty\}$
is the {\it pressure function}.
%
%
Let $\rho_{\s}$ and $\sigma_{\s}$ be, respectively, the Ricci tensor and the scalar curvature of $\Lr$. Then the Einstein field equation\footnote{We use units in which the gravitational constant $G=1$ and the speed of light $c=1$.}
\[
  \rho_{\s}+\left(\Lambda-\frac{1}{2}\sigma_{\s}\right)\Lr=8\pi \mathrm{T}
   \]
is equivalent to
\begin{equation}\label{EF}
\begin{cases}
 \m = \frac{3}{8\pi}\left(\frac{\dot{\s}^2}{\s^2}+\frac{1}{\s^2}-\frac{\Lambda}{3}\right),\\
  \p = -\frac{1}{4\pi \s}\left(\ddot{\s}+\frac{1}{2\s}(\dot{\s}^2+1)-\frac{\Lambda}{2}\s\right),
   \end{cases}
    \end{equation}
where dot denotes differentiation with respect to the cosmological time $t$.
We implicitly assume that $\I$ is the maximal domain of definition of the scale function,
i.e., $\s$ cannot be extended to a nonnegative differentiable function whose interval
of definition contains properly $\I$.

\begin{remark}
The RW spacetime with $\I=\R$, $\s=1$, $\m+\p=(4\pi)^{-1}$, and $\Lambda = 1+8\pi\p$
is the {\it Einstein static universe} (cf. \cite{Einstein}), which will be denoted by $(\E,\LrE)$.
\end{remark}

\begin{defn}
There are three standard physical conditions that are usually imposed on a RW spacetime (cf. \cite{HE}):
\begin{itemize}
\item the {\it weak energy condition} : $-\m\le \p$;
\item the {\it dominant energy condition} : $-\m\le \p\le \m$;
\item the {\it strong energy condition} : $-\m/3\le \p\le \m$.
\end{itemize}
\end{defn}

\begin{assumption}
If, on the one hand, the scale function $\s$ is allowed to vanish at some $t\in \mathrm I_\s$,
we will assume that the {\it conformal time}
\begin{equation}\label{CT}
%
\tau_{\mathfrak s} :  \I  \to \mathbb R,\, t \mapsto \tau_{\mathfrak s}(t) :=\int_{t_0}^t \frac{du}{\s(u)}
\end{equation}
is a continuous, strictly increasing function.
\end{assumption}

 Under the previous assumption, $\J :=\tau_{\s}(\I)$ is an open interval
and $\tau_{\s}^{-1}:\J\to \R$ is a strictly increasing differentiable function. Moreover, the map
\begin{equation}\label{ce}
 \F :  \J\times \SS \ni (\tau,p) \mapsto (\tau_{\s}^{-1}(\tau),p)\in \RW.
  \end{equation}
is a differentiable homeomorphism, such that
\[
 \F^*(\Lr) = \S^2(-dt^2+g),
  \]
where $\S:=\s\circ \tau_{\s}^{-1}:\J\to \R$ is a nonnegative differentiable function, called
the {\it conformal scale function}.
Note that $\J$ is the maximal domain of definition of $\S$. Therefore, $\F$ is a
weakly conformal homeomorphism of the open domain $\J\times \mathbb S^3$ of the Einstein static universe onto $\RW$.
The map $\F$ fails to be of maximal rank at the points $(\tau,p)\in \J\times \SS$, such that $\S(\tau)=0$.
The inverse map $\F^{-1}$ is a homeomorphism which fails to be differentiable at
the points $(t,p)\in \I\times \SS$, such that $\s(t)=0$.

\begin{defn}\label{def:cyclic}
Let $\LrC$ denote the quadratic form $\F^*(\Lr)$ and put $\widetilde{\RW}=\J\times \mathbb S^3$. We say that
$(\widetilde{\RW}, \LrC)$ is the {\it conformal model} of the RW spacetime with scale function $\s$.
If $\J=\R$ and $\S$ is periodic, we say that the conformal model $(\widetilde{\RW}, \LrC)$ is {\it cyclic}.
The functions
$\cm:=\m\circ \tau_{\mathfrak s}^{-1}:\J\to \R\cup\{+\infty\}$ and
$\cp:=\p\circ \tau_{\mathfrak s}^{-1}:\J\to \R\cup\{\pm\infty\}$
are called the {\it conformal energy density} and the {\it conformal pressure}, respectively.
\end{defn}

In view of \eqref{EF} and the fact that,  by \eqref{CT}, $d\tau = dt/\mathfrak s$,
the conformal energy density and the conformal pressure can be written as
\begin{equation}\label{cmcp} 
\begin{cases}
\widetilde{\m}=\frac{3}{8\pi}\left(\frac{\S'^2}{\S^4}+\frac{1}{\S^2}-\frac{\Lambda}{3}\right),\\
{\cp =-\frac{1}{4\pi \S}\left(\frac{\S''}{\S^2}-\frac{\S'^2}{2\S^3}+\frac{1}{2\S}-\frac{\Lambda \S}{2} \right)}.
\end{cases}
\end{equation}
Here prime denotes differentiation with respect to the conformal time $\tau$.


\subsection{$\m$-critical RW universes and their state equations}

The total energy of the compact domain $\mathcal{D}=[t_0,t_1]\times \mathbb S^3\subset \RW$ is
\[
  \mathfrak{M}=\int_{\mathcal{D}}\m \,dV_{\Lr},
    \]
where $dV_{\Lr}$ is the volume element of $\Lr$.

\begin{defn}
A RW spacetime is said to be {\it $\m$-critical} if its scale function is a critical point of the functional
$\mathfrak{M}$ with respect to volume-preserving variations of the RW quadratic form $\Lr$.
\end{defn}

\begin{prop}\label{Proposition1}
A RW spacetime is $\m$-critical if and only if its scale function $\s$ satisfies
\begin{equation}\label{VE2}
 \s\left(\dot{\s}^2+\frac{\mathfrak{a}}{3}\s^2-1\right)=\H,
  \end{equation}
where $\mathfrak{a}$ and $\H$ are two arbitrary constants, called the {\em internal parameters}.
In addition, the state equation of an $\m$-critical spacetime with parameters $\mathfrak{a}$ and $\H$ is
\begin{equation}\label{EQS}
 \frac{8\pi}{3}\m+8\pi \p -\frac{2}{3}(\Lambda+\mathfrak{a})-\frac{\H}{2\sqrt{2}}(\Lambda+\mathfrak{a}-8\pi \p)^{3/2}=0.
  \end{equation}
   \end{prop}

\begin{proof}
The volume element of $\Lr$ is $\s^3dt\wedge d\upsilon$, where $d\upsilon$ is the volume element of the unit
3-sphere $\SS$. We then have
\[
   \mathfrak{M}= \frac{3\pi}{4}\int_{t_0}^{t_1}\s(\dot{\s}^2-\frac{\Lambda}{3}\s^2+1)dt.
     \]
The constraint on the volume amounts to requiring that
\begin{equation}\label{con}
   \int_{t_0}^{t_1}\s^3 dt = \mathrm{const}.
     \end{equation}
Thus, a RW spacetime is $\m$-critical if and only if its scale function $\mathfrak s$ is
critical for the action integral $\int_{t_0}^{t_1}L(\s, \dot{\s})dt$ corresponding to the Lagrangian
\[
  L(\s,\dot{\s})=\s(\dot{\s}^2-\frac{\Lambda+\lambda}{3}\s^2+1),
    \]
where $\lambda\in \R$ is the Lagrange multiplier associated to the constraint \eqref{con} (see \cite{BS}),
Equivalently,
$\s$ is a solution of the Euler--Lagrange equation of the action,
\begin{equation}\label{VE}
 \frac{d}{d t}\frac{\partial L}{\partial \dot{\s}}-\frac{\partial L}{\partial \s} = 2\s\ddot{\s}+\dot{\s}^2+\mathfrak{a}\s^2-1=0,
   \end{equation}
where $\mathfrak{a}:=\Lambda+\lambda$.
Now, since $L$ does not depend explicitly on time, taking the total time derivative of $L$
and replacing $\frac{\partial L}{\partial \s}$ by $\frac{d}{d t}\frac{\partial L}{\partial \dot{\s}}$,
in accordance with the Euler--Lagrange equation, yields
\[
\frac{d L}{d t} = \frac{\partial L}{\partial {\s}} \dot{\s} + \frac{\partial L}{\partial \dot{\s}}\ddot{\s}=
\frac{d}{d t}\left( \dot{\s} \frac{\partial L}{\partial \dot{\s}} \right).
\]
This implies that
\[
 \dot{\s} \frac{\partial L}{\partial \dot{\s}} -L = \s(\dot{\s}^2+\frac{\mathfrak{a}}{3}\s^2-1)
  \]
is a first integral of the motion.
%
%
Therefore, $\s$ satisfies \eqref{VE} if and only if there exists a constant $\H$ such that
\begin{equation}\label{VE21}
 \s(\dot{\s}^2+\frac{\mathfrak{a}}{3}\s^2-1)=\H.
   \end{equation}
In view of \eqref{VE}, the first equation of \eqref{EF} becomes
\begin{equation}\label{VE22}
   \s^2=\frac{2}{\Lambda+\mathfrak{a}-8\pi \p}.
     \end{equation}
Thus $8\pi \p\le (\Lambda+\mathfrak{a})$. Using \eqref{EF}, \eqref{VE2} and \eqref{VE22} it is now an
easy matter to check
that the equation of state of an $\m$-critical RW spacetime, with internal parameters $\H$ and $\mathfrak{a}$,
and cosmological constant $\Lambda$, is
\begin{equation}\label{EQS1}
 \frac{8\pi}{3}\m+8\pi \p -\frac{2}{3}(\Lambda+\mathfrak{a})-\frac{\H}{2\sqrt{2}}(\Lambda+\mathfrak{a}-8\pi \p)^{3/2}=0.
  \end{equation}
This concludes the proof.
\end{proof}

\begin{remark}
If $\H=0$, then \eqref{VE2} can be easily integrated in terms of elementary functions. As a result we find
\begin{equation}
\begin{split}
\s(t)&=\vert\sqrt{{3}/{\a}}\,\sin\big(\sqrt{{\a}/{3}}\,t+c\big)\vert, \quad \a>0,\\
\s(t)&=\vert\sqrt{{3}/{\a}}\,\sinh\big(\sqrt{{\a}/{3}}\,t+c\big)\vert,\quad \a<0,\\
\s(t)&=\vert t+c\vert,\quad  \a=0.
\end{split}
\end{equation}
%
From now on, we will assume that $\H$ is different from zero and $\s$ is nonconstant.
\end{remark}

\subsection{The conformal scale function of a critical RW spacetime}

\begin{defn}
A {\it real form} of the Weierstrass elliptic function
$\wp (z;g_2,g_3)$
with \textit{real} invariants $g_2$ and $g_3$ is a nonconstant function
$\Q:\R\to \R\cup\{+\infty\}$, such that
\begin{equation}\label{wp}
  \Q'^2=4\Q^3-g_2Q-g_3.
    \end{equation}
\end{defn}

\begin{remark}\label{realforms}
The behavior of these functions depends on the discriminant $\Delta = 16(g_2^3-27g_3^2)$ of the cubic polynomial $P(x)=4x^3-g_2x-g_3$.

\vskip0.2cm

 If $\Delta=0$ (degenerate case), then
\begin{itemize}
\item if $g_3=-8a^3>0$, then $\Q(\tau)=-3a\tan(\sqrt{-3a}\tau+c)^2-2a$, $c\in \R$;
\item if $g_3=-8a^3>0$, then either
\begin{enumerate}
\item $\Q(\tau)=3a\tanh(\sqrt{3a}\tau+c)^2-2a$, $c\in \R$, or
\item $\Q(\tau)=a(1+3\mathrm{csch}(\sqrt{3a}\tau+c)^2)$, $c\in \R$;
\end{enumerate}
\item if $g_3=g_2=0$, then $\Q(\tau)=(\tau+c)^2$, $c\in \R$.
\end{itemize}

\vskip0.2cm

If $\Delta \neq 0$, the Weierstrass $\wp$-function is doubly periodic, with primitive half-periods
$\omega_1$, $\omega_3$, such that $\omega_1\in \R$, $\omega_1>0$
and $\mathrm{Im}(\omega_3/\omega_1)>0$.
If $\Delta>0$, $\omega_3$ is purely imaginary, while, if $\Delta<0$, $\mathrm{Re}(\omega_3)\neq 0$.
Its real forms are:
\begin{itemize}
\item if $\Delta>0$, then either
\begin{enumerate}
\item $\Q(\tau)=\wp(\tau+2m\omega_1+2n\omega_3+c,g_2,g_3)$, $c\in\R$, $m,n\in \Z$, or
\item $\Q(\tau)=\wp(\tau + \omega_3+2m\omega_1+2n\omega_3+c,g_2,g_3)$, $c\in\R$, $m,n\in \Z$;
\end{enumerate}
\item if $\Delta<0$, then $\Q(\tau)=\wp(\tau+2m\omega_1+2n\omega_3+c,g_2,g_3)$, $c\in \R$, $m,n\in \Z$.
\end{itemize}
The constants $m,n$ and $c$ are irrelevant and can be put equal to $0$.
\end{remark}

We can prove the following.

\begin{prop}\label{Proposition2}
The conformal scale function of an $\m$-critical RW universe with internal parameters $\H$ and $\mathfrak{a}$, and cosmological constant $\Lambda$, is given by
\begin{equation}\label{Scale}
  \S(\tau)=\frac{3\H}{12\Q(\tau)-1},
   \end{equation}
where $\Q$ is a real form of the Weierstrass $\wp$-function with invariants
\begin{equation}\label{inv}
  g_2=\frac{1}{12}, \quad g_3=\frac{1}{48}\Big(\a\H^2-\frac{2}{9}\Big).
   \end{equation}
The conformal energy density and the conformal pressure are given by
\begin{equation}\label{cmp}
\begin{split}
\widetilde{\m} &={\frac{1}{8\pi}\left(\frac{(1-12\Q)^2(5+12\Q)}{9\H^2}-\mathfrak{a}-\Lambda \right)},
  \\
\widetilde{\p} &=\frac{1}{4\pi} \left(\frac{\mathfrak{a}+\Lambda}{2}-\frac{(12\Q-1)^2}{9\H^2}\right).
\end{split}
   \end{equation}
\end{prop}

\begin{proof}
Differentiating $\S=\s\circ \tau_{\s}^{-1}$ and using \eqref{CT} we have
\begin{equation}\label{ds}
 \dot{\s}|_{\tau_{\s}^{-1}(\tau)}=\S^{-1}(\tau)\S'|_{\tau},
  \end{equation}
where $\dot{f}$ is the derivative with respect to the cosmological time and $h'$ the derivative with respect to the conformal time. Therefore, if $\H\neq 0$, the variational equation \eqref{VE2} is satisfied if and only if
\begin{equation}\label{VE3}
  \S'^2+\frac{\mathfrak{a}}{3}\S^4-\S^2-\H\S=0.
   \end{equation}
Let
\begin{equation}\label{Q}
  \Q:=\frac{\H}{4\S}+\frac{1}{12}.
    \end{equation}
It is an easy matter to check that \eqref{VE3} holds true if and only if
\begin{equation}\label{VE4}
 \Q'^2=4\Q^3-\frac{1}{12}\Q+\frac{1}{48}\Big(\frac{2}{9}-\a\H^2\Big).
   \end{equation}
This proves the first part of the statement.
Substituting \eqref{VE3} into the first equation of \eqref{cmcp} and taking into account \eqref{Scale},
we have
\[
\begin{split}
\widetilde{\m}&= \frac{3}{8\pi}\left[\frac{1}{\S^4} \left(-\frac{\mathfrak{a}}{3}\S^4+\S^2+\H\S\right)
 +\frac{1}{\S^2}- \frac{\Lambda}{3}\right]\\
&=\frac{3}{8\pi}\left(\frac{2}{\S^2}+\frac{\H}{\S^3}-\frac{\mathfrak{a}+\Lambda}{3} \right)\\
&={\frac{1}{8\pi}\left(\frac{(1-12\Q)^2(5+12\Q)}{9\H^2}-\mathfrak{a}-\Lambda \right).}
\end{split}\]
Differentiating \eqref{VE3} yields
\begin{equation}\label{VE5}
\S''=-\frac{2\mathfrak{a}}{3}\S^3 + S + \frac{\H}{2}. 
  \end{equation}
Substituting \eqref{VE3} and \eqref{VE5} into the second equation of \eqref{cmcp}, we obtain
\[
{\begin{split}
\widetilde{\p}&=-\frac{1}{4\pi\S}\left[\frac{1}{\S^2}\Big(-\frac{2\mathfrak{a}}{3}\S^3+\S +\frac{\H}{2}\Big)
 -\frac{1}{2\S^3}\Big(-\frac{\mathfrak{a}}{3}\S^4+\S^2+\H S\Big)+\frac{1}{2\S}-\frac{\Lambda}{2} \S \right] \\
&=\frac{1}{4\pi} \left(\frac{\mathfrak{a}+\Lambda}{2}-\frac{1}{\S^2}\right)=
\frac{1}{4\pi} \left(\frac{\mathfrak{a}+\Lambda}{2}-\frac{(12\Q-1)^2}{9\H^2}\right).
\end{split}
}
\]
This concludes the proof.
\end{proof}

\begin{remark}
The function $\tau_{\s}^{-1}$ is given by the incomplete elliptic integral of the third kind
\begin{equation}\label{ct}
 \tau_{\s}^{-1}(\tau)=3\H\int \frac{d\tau}{12\Q(\tau)-1}.
  \end{equation}
Actually, this integral can be computed in closed form using the $\sigma$ and $\zeta$ Weierstrass functions
(see \cite{La}, p. 173). The explicit expression of $\tau_\s$ cannot be given in closed form. However,
it can be evaluated using numerical solutions of the first order ODE $h'=3\H^{-1}(12\Q(h)-1)$.
\end{remark}

 If $g_2$ and $g_3$ are as in \eqref{inv}, the discriminant $\Delta =16(g_2^3-27g_3^2)$ takes the form
\[
   \Delta = {\frac{1}{48}}\a\H^2(4-9\a\H^2). 
   \]

\begin{defn}
According to whether $\Delta \neq 0$ or $\Delta = 0$, an $\m$-critical RW universe is said to be
\textit{general} or \textit{exceptional}. Moreover, a general $\m$-critical RW universe
is said to be of Class I, if $\Delta >0$,
and of Class II, if $\Delta <0$.

\end{defn}


%
%

\subsection{Exceptional $\m$-critical RW universes}\label{ss:exceptional}
The conformal scale functions of the exceptional $\m$-critical RW universes can be written in terms of trigonometric
or hyperbolic functions.
\begin{itemize}
\item If $\H^2\a=4/9$, then $\H>0$ and
\[
  \S(\tau)=\frac{3\H}{3\tan^2({\tau}/{2})+1}.
  \]
The quadratic form $\ell_{\S}$ vanishes if $\tau = (2h+1)\pi$, $h\in \Z$, and is of Lorentz type at all other
points. The conformal energy density function $\cm$ is nonnegative if and only if $\Lambda\le 1/3\H^2$.
Assuming $\cm\ge 0$, the weak energy condition
is automatically fulfilled, while the dominant energy condition holds true if and only if $\Lambda\le 1/18\H^2$.
The strong energy condition is satisfied if and only if {$-1/2\H^2$}$\le \Lambda \le 1/18\H^2$.

\item If $\a=0$ and $\H<0$, then
\[
   \S(\tau)=-\frac{\H}{2}(1+\cosh \tau ).
   \]
The quadratic form $\ell_{\S}$ is nondegenerate of Lorentz type and $\cm$ is nonnegative
if and only if $\Lambda\le 0$.
Under this assumption the weak and the dominant energy conditions are automatically satisfied.
The strong energy condition is never satisfied.

\item If $\a=0$ and $\H>0$, then
\[
  \S(\tau)=\H\sinh^2(\tau/2).
   \]
Then $\ell_{\S}$ vanishes if $\tau=0$ and is of Lorentz type if $\tau \neq 0$. The conformal energy
density $\cm$ is nonnegative if and only
if $\Lambda\le 0$. Under this hypothesis, the weak and the dominant energy conditions hold true, while the strong
energy condition is satisfied if and only if $\Lambda=0$.
\end{itemize}

\section{Critical RW universes of Class I
}\label{s:2}

This section discusses $\m$-critical RW universes of Class I.
Theorem \ref{Teorema1} describes the main features of the critical models of Class I
with negative $\H$, while Theorem \ref{Teorema2} discusses the critical models of Class I with positive $\H$.

\vskip0.2cm

If $\Delta>0$, the cubic polynomial $P(x)$ has three distinct real roots, say $e_1>e_2>e_3$, such that
\[
   e_1>0,    \quad e_3=-(e_1+ e_2)<0,   \quad 4(e_1^2 + e_1e_2 + e_2^2) = 1/12.
     \]
We write $e_1$ and $e_2$ as functions of the {\it angular parameter} $\vartheta\in (-\pi/3,0)$,
\begin{equation}
\begin{aligned}
e_1(\vartheta) &=\frac{1}{4\sqrt{3}}\left(\frac{1}{\sqrt{3}}\cos \vartheta-\sin \vartheta\right),\\
e_2(\vartheta) &=\frac{1}{4\sqrt{3}}\left(\frac{1}{\sqrt{3}}\cos \vartheta +\sin \vartheta\right).
\end{aligned}
\end{equation}
The invariant $g_3$ and the parameter $\mathfrak{a}$ can be written as functions of $\H$ and $\vartheta$, namely
\begin{equation}\label{inv1}
   g_3 = - \frac{\cos(3\vartheta)}{216},\quad \a = \frac{2}{9\H^2}\left(1-\cos 3 \vartheta \right).
      \end{equation}
Let $\wp_{\vartheta}$ be the Wierstrass $\wp$-function with invariants $g_2=1/12$ and $g_3({\vartheta})$.
Let $\omega_{\vartheta}$ and  $\omega'_{\vartheta}$ be its real and purely imaginary half-periods. Let
\[
 \widetilde{\wp}_{\vartheta}(\tau)=\wp_{\vartheta}(\tau+\omega'_{\vartheta}),\quad
  \widehat{\wp}_{\vartheta}(\tau)=\wp_{\vartheta}(\tau),\quad \forall \tau\in \R,
  \]
be the real forms of $\wp_{\vartheta}$.

\subsection{Critical models of Class I with negative $\mathcal H$}

With the notation introduced above,  let $f_1$, $f_2$ and $f_3$ be defined by
\begin{eqnarray*}
f_1(\vartheta,\H)&=&\frac{1}{3\H^2}\left(1-\cos \vartheta - \sqrt{3}\sin \vartheta\right)^2,\\
f_2(\vartheta,\H)&=&\frac{2}{9\H^2}\sin^2\Big(\frac{\vartheta}{2}\Big)\left(7+\cos\vartheta-2\cos 2\vartheta
- 5\sqrt{3}\sin\vartheta\right),\\
f_3(\vartheta,\H)&=&-\frac{1}{6\H^2}(1-2\cos\vartheta)(1+2\cos\vartheta)^2.
\end{eqnarray*}
{Note that
\[
  f_3(\vartheta,\H)\le f_2(\vartheta,\H),\quad \forall \,\,(\vartheta,\H)\in (-\pi/3,\vartheta_I]\times (-\infty,0),
  \]
where $\vartheta_I\approx -0.7706314502$ is the root in the interval $(-\pi/3,0)$ of the equation
\[
  -10 + 22\cos{x}+11\cos{2x} + 4\cos{3x} + 10\sqrt{3}\sin{x}-5\sqrt{3}\sin{2x}=0.
  \]
}

We can now state the following.

\begin{thm}\label{Teorema1}
Let $\mathcal M(\s)$ be
an $\m$-critical RW spacetime of Class I, with cosmological constant $\Lambda$, angular parameter
$\vartheta\in (-\pi/3,0)$, and negative $\H$. Let $(\RWC_{\H,\vartheta}, \ell_{\H,\vartheta})$ denote
its conformal model.
Then,
\[
  \Lambda \le f_1(\vartheta,\H),\quad \RWC_{\H,\vartheta}=\R\times \SS,
  \]
and $\ell_{\H,\vartheta}$ is the cyclic nondegenerate Lorentz metric
\begin{equation}\label{clr11}
  \ell_{\H,\vartheta}= \frac{9\H^2}{(12\widetilde{\wp}_{\vartheta}-1)^2}(-d\tau^2+g).
   \end{equation}
The weak energy condition is automatically satisfied. The dominant energy condition is satisfied if and only if
$\Lambda\le f_2(\vartheta,\H)$. The strong energy condition is satisfied if and only if
$\vartheta \in (-\pi/3,\vartheta_I]$ and $f_3(\vartheta,\H)\le \Lambda\le f_2(\vartheta,\H)$.
\end{thm}

\begin{proof}
According to Proposition \ref{Proposition1}, we have
$\S=3\H/(12\Q-1)$, where $\Q$ is one of the two possible real forms of $\wp_{\vartheta}$.
Since {$\S \geq 0$} and $\H<0$, it follows that
$\Q$ is bounded above by $1/12$. The roots
\[
  e_1(\vartheta)>e_2(\vartheta)>e_3(\vartheta)=- e_1(\vartheta)-e_2(\vartheta)
   \]
of the polynomial $4t^3-t/12-g_3$ satisfy
\[
  e_3(\vartheta)< e_2(\vartheta)  <  \frac{1}{12} <e_1(\vartheta).
   \]
Since
\[
   e_3(\vartheta)\le \widetilde{\wp}_{\vartheta}(\tau)\le e_2(\vartheta),
   \quad e_1(\vartheta)\le \widehat{\wp}_{\vartheta}(\tau),
    \]
$\Q$ must coincide with $\widetilde{\wp}_{\vartheta}$. This implies that $\S$ is a strictly positive,
real analytic periodic function with minimal period $2\omega_{\vartheta}$. Thus, the conformal model
is $\R\times \SS$ and $\ell_{\S}$ coincides with \eqref{clr11}.

It follows from \eqref{cmp} that the conformal energy density and the conformal pressure are given by
\begin{eqnarray*}
\widetilde{\m} &=&\frac{1}{8\pi}\left(\frac{(1-12\widetilde{\wp}_{\vartheta})^2(5+12
\widetilde{\wp}_{\vartheta})}{9\H^2}-\mathfrak{a}-\Lambda\right),\\
\widetilde{\p} &=&\frac{1}{4\pi} \left(\frac{\mathfrak{a}+\Lambda}{2}-
  \frac{(12\widetilde{\wp}_{\vartheta}-1)^2}{9\H^2}\right),
   \end{eqnarray*}
where $\mathfrak{a}$ is as in \eqref{inv1}.
Thus, $\cm$ and $\cp$ are real analytic periodic functions of
period $2\omega_{\vartheta}$.
The function $\cm$ achieves its maxima at the points $\tau^+_h=2h\omega_{\vartheta}$, $h\in \Z$, and
its minima at $\tau^-_h=(2h+1)\omega_{\vartheta}$, $h\in \Z$.
On the other hand, $\cp$ achieves
its maxima at $\tau^-_h=(2h+1)\omega_{\vartheta}$, $h\in \Z$, and
its minima at $\tau^+_h=2h\omega_{\vartheta}$, $h\in \Z$.

 Taking into account that
\[
  \widetilde{\wp}_{\vartheta}(\omega_{\vartheta})= e_2(\vartheta),
   \quad  \widetilde{\wp}_{\vartheta}(0)= -e_1(\vartheta)- {e_2}(\vartheta),
     \]
it follows that
\begin{equation}\label{a}
\begin{aligned}
\cm(0)&=\frac{3 - 3 \H^2\Lambda + 4 \cos \vartheta +2 \cos 2\vartheta}{24\H^2\pi},\\
\cp(0)&= \frac{9\H^2\Lambda-2 (1 + 2\cos \vartheta)^2\cos\vartheta}{72\H^2\pi},
 \end{aligned}
  \end{equation}
and that
\begin{equation}\label{b}
 \begin{aligned}
\cm(\omega_{\vartheta})&=\frac{1}{24\H^2\pi}\left(\left(1-\cos \vartheta
-\sqrt{3}\sin \vartheta\right)^2  -3\H^2\Lambda   \right),\\
\cp(\omega_{\vartheta})&= \frac{1}{72\H^2\pi}\left(9\H^2\Lambda-4+4\cos \vartheta
+2\cos 2\vartheta -2\cos 3\vartheta
\right.\\
& \qquad\qquad\qquad +\left.  4\sqrt{3}(1-\cos \vartheta)\sin \vartheta\right).
\end{aligned}
 \end{equation}
From \eqref{b}, it follows that
\[
  \min (\widetilde{\m})=\frac{1}{24\H^2\pi}\left(\left(1-\cos\vartheta
   -\sqrt{3}\sin\vartheta\right)^2 -3\Lambda \H^2\right).
     \]
This implies that $\widetilde{\m}\ge 0$ if and only if $\Lambda \le f_1(\vartheta,\H)$, as claimed.

The function $\cm+\cp$ is real analytic and periodic, with period $2\omega_{\vartheta}$, and attains
its minimum at the points $\tau^-_h$.
%
In particular, we compute
\[
  \min(\cm+\cp) = \frac{1}{18\H^2\pi}\sin^2\Big(\frac{\vartheta}{2}\Big)\left(5+5\cos \vartheta+
        2\cos 2\vartheta - \sqrt{3} \sin \vartheta\right).
      \]
Now, the right-hand-side of the previous equation is strictly positive for every $\vartheta\in (-\pi/3,0)$,
which implies that the weak energy condition is automatically satisfied.
Similarly, we have
\[
  \min(\cm-\cp) =-\frac{\Lambda}{4\pi}+\frac{\sin^2(\vartheta/2)}{18\H^2\pi}
    \left(7+\cos \vartheta -2\cos 2\vartheta - 5\sqrt{3} \sin \vartheta\right).
   \]
Thus $\min(\cm-\cp)\ge 0$ if and only if $\Lambda \le f_2(\vartheta,\H)$.
Since $f_2(\vartheta,\H)<f_1(\vartheta,\H)$, for each $\H<0$ and for each $\vartheta\in (-\pi/3,0)$,
it follows that the dominant energy condition is satisfied if and only if $\Lambda \le f_2(\vartheta,\H)$.

 The function $\cp+\cm/3$ is real analytic and periodic, with period $2\omega_{\vartheta}$, and attains
 its minimum at the points $\tau^+_h$. From \eqref{a}, we have
\[
  \min(\cp+\cm/3) =\frac{\Lambda}{12\pi}+
\frac{1}{72\H^2\pi}(1-2\cos \vartheta)(1+2\cos \vartheta)^2.
  \]
Thus, the strong energy condition is satisfied if and only if
$\vartheta \in (-\pi/3,\vartheta_I]$ and $f_3(\vartheta,\H)\le \Lambda \le f_2(\vartheta,\H)$. This concludes the proof.
\end{proof}


\subsection{Critical models of Class I with positive $\mathcal H$}

With the notation introduced above, let $\widehat{f}_1$, $\widehat{f}_2$ and $\widehat{f}_3$ be given by
\begin{eqnarray*}
\widehat{f}_1(\vartheta,\H)&=&\frac{1}{3\H^2}(1-\cos \vartheta + \sqrt{3}\sin \vartheta)^2,\\
\widehat{f}_2(\vartheta,\H)&=&\frac{2}{9\H^2}\sin^2\Big(\frac{\vartheta}{2}\Big)\left(7+\cos \vartheta - 2\cos 2\vartheta
+5\sqrt{3}\sin \vartheta\right),\\
\widehat{f}_3(\vartheta,\H)&=&\frac{1}{6\H^2}\left(1-2\cos \vartheta - \cos 2\vartheta+
2\cos 3\vartheta+
2\sqrt{3}\sin \vartheta-\sqrt{3}\sin 2\vartheta\right).
\end{eqnarray*}
Note that
\[
 \widehat{f}_3(\vartheta,\H)<0<\widehat{f}_2(\vartheta,\H)<\widehat{f}_1(\vartheta,\H),
  \quad \forall \,\, (\H, \vartheta) \in (0,+\infty)\times (-\pi/3,0).
    \]

We can prove the following.

\begin{thm}\label{Teorema2}
Let $\mathcal M(\s)$ be
an $\m$-critical RW space time of Class I, with cosmological constant $\Lambda$,
angular parameter $\vartheta\in (-\pi/3,0)$, and positive $\H$.
Let $(\RWC_{\H,\vartheta}, \ell_{\H,\vartheta})$ denote its conformal model. Then,
\[
 \Lambda \le \widehat{f}_1(\vartheta,\H),\quad \RWC_{\H,\vartheta}=\R\times \SS,
   \]
and $\ell_{\H,\vartheta}$ is the cyclic quadratic form
\begin{equation}\label{clr12}
  \ell_{\H,\vartheta}=
  \frac{9\H^2}{(12\widehat{\wp}_{\vartheta}-1)^2}(-d\tau^2+g).
    \end{equation}
The form $\ell_{\H,\vartheta}$ vanishes along the totally umbilical spacelike hypersurfaces
$$\mathbb{S}^3_h=\{(\tau,p)\in \RWC_{\H,\vartheta}
\mid \tau = 2h\omega_{\vartheta}\},$$ $h\in \Z$, and is
nondegenerate and of Lorentz type on the complement of $\bigcup_{h\in \Z}\mathbb{S}^3_h$.
The weak energy condition is automatically satisfied. The dominant energy condition is satisfied if and only if
$\Lambda\le \widehat{f}_2(\vartheta,\H)$. The strong energy condition is satisfied if and only if
$\widehat{f}_3(\vartheta,\H) \leq \Lambda \leq \widehat{f}_2(\vartheta,\H)$.
\end{thm}

\begin{proof}
Arguing as in the first part of the proof of Theorem \ref{Teorema1}, one sees that the function $\Q$ must
coincide with $\widehat{\wp}_{\vartheta}$, and hence the conformal scale function
\[
   \S(\tau)=\frac{3\H}{12\widehat{\wp}_{\vartheta}(\tau)-1}.
     \]
Then, $\S$ is nonnegative, periodic, with minimal period $2\omega_{\vartheta}$, and has zeroes of second
order at the points $2h\omega_{\vartheta}$, $h\in \Z$. This implies that $\ell_{\H,\vartheta}$
is as in \eqref{clr12}. Hence, $\ell_{\H,\vartheta}$ is defined on $\RWC_{\H,\vartheta}=\R\times \SS$,
is nondegenerate and of Lorentz type on the strips
$(2h\omega_{\vartheta},2(h+1)\omega_{\vartheta})\times \SS$, $h\in \Z$, and vanishes
when $\tau = 2h\omega_{\vartheta}$.
Proceeding as in the proof of the previous theorem, one deduces that
\begin{equation}\label{ccpm}
 \begin{aligned}
  \widetilde{\m} &=\frac{1}{8\pi}\left(\frac{(1-12\widehat{\wp}_{\vartheta})^2(5+12
   \widehat{\wp}_{\vartheta})}{9\H^2}-\mathfrak{a}-\Lambda\right),\\
   \widetilde{\p} &=\frac{1}{4\pi} \left(\frac{\mathfrak{a}+\Lambda}{2}-
    \frac{(12\widehat{\wp}_{\vartheta}-1)^2}{9\H^2}\right),
     \end{aligned}
      \end{equation}
where $\mathfrak{a}$ is as in \eqref{inv1}.
Thus, $\cm$ and $\cp$ are periodic function of period $2\omega_{\vartheta}$, real analytic on the intervals $(2h\omega_{\vartheta},2(h+1)\omega_{\vartheta})$, with poles of finite order at $2h\omega_{\vartheta}$,
$h\in \mathbb Z$.

The value of $\cm$ is minimum at $\tau_h=(2h+1)\omega_{\vartheta}$, $h\in \Z$.
Since $\widehat{\wp}_{\vartheta}(\omega_{\vartheta})= e_1(\vartheta)$, we have
\[
 \min(\widetilde{\m})=
  \frac{1}{24\pi\H^2}\left(-3\H^2\Lambda+\left(1-\cos \vartheta +\sqrt{3}\sin \vartheta\right)^2\right).
   \]
Thus $\cm\ge 0$ if and only if $\Lambda\le \widehat{f}_1(\vartheta,\H)$.
Similarly, $\cm+\cp$, $\cm-\cp$, and $\cm/3+\cp$ are periodic with period $2\omega_{\vartheta}$, real analytic on $(2h\omega_{\vartheta},2(h+1)\omega_{\vartheta})$, and have poles of finite order when
$\tau = 2h\omega_{\vartheta}$. They achieve their minima at $\tau_h=(2h+1)\omega_{\vartheta}$, $h\in \Z$. This implies that
\[
 \min(\widetilde{\m}+\cp)=\frac{\sin^2(\vartheta/2)}{18\pi \H^2}\left(5+5\cos \vartheta
   +2\cos 2\vartheta + \sqrt{3}\sin \vartheta\right).
   \]
Since the right-hand-side of the previous equation is positive, for each $\H>0$ and for each $\vartheta \in (-\pi/3,0)$,
the weak energy condition is automatically satisfied. Similarly, we have
\[
 \min(\widetilde{\m}-\cp)=\frac{\sin^2(\vartheta/2)}{18\pi \H^2}\left(7+\cos \vartheta
    +2\cos 2\vartheta+5\sqrt{3}\sin \vartheta\right)-\frac{\Lambda}{4\pi}.
   \]
Therefore, $\widetilde{\m}-\cp$ is a nonnegative function if and only if $\Lambda \le \widehat{f}_2(\vartheta,\H)$.
On the other hand, $\widehat{f}_2(\vartheta,\H)<\widehat{f}_1(\vartheta,\H)$, for every $\H$ and for every
$\vartheta\in (-\pi/3,0)$. This implies that the dominant energy condition is satisfied if and only if $\Lambda \le \widehat{f}_2(\vartheta,\H)$.
Finally, the minimum of $\cm/3+\cp$ is
\[
 \begin{split}\min(\widetilde{\m}/3+\cp)=&\frac{1}{72\pi \H^2}\left(-1+2\cos \vartheta
  +\cos 2\vartheta - 2\cos 3\vartheta \right.\\
   &\qquad\qquad \left. -2\sqrt{3}\sin \vartheta+
    \sqrt{3}\sin 2\vartheta \right)+\frac{\Lambda}{12\pi}.
     \end{split}
       \]
Hence, $\widetilde{\m}/3+\cp$ is nonnegative if and only if $\widehat{f}_3(\vartheta,\H)\le \Lambda$.
The strong energy condition is satisfied if and only if
$\widehat{f}_3(\vartheta,\H)\le \Lambda\le \widehat{f}_2(\vartheta,\H)$.
\end{proof}

\section{Critical RW universes of Class II
}\label{s:3}

This section discusses the $\m$-critical RW universes of Class II. Two possible types are distinguished:
negative and positive types. Theorem \ref{Teorema3} provides a complete description of the critical
models of negative type. Theorem \ref{Teorema4} deals with the critical models of positive type with
$\H>0$, while Theorem \ref{Teorema5} describes the main features of the critical
models of positive type with $\H<0$.

\vskip0.2cm

If $\Delta<0$, the cubic polynomial $P(x)$ has two complex conjugate roots, $e_1+ie_2$, $e_1-ie_2$, $e_2>0$,
and a real root, $-2e_1$, such that $144e_1^2-48e_2^2=1$. Consequently, we can write
\begin{equation}\label{TI}
 e_1=\frac{\varepsilon}{12}\cosh \vartheta,
 \quad e_2=\frac{1}{4\sqrt{3}}\sinh \vartheta,\quad \vartheta >0,
  \quad \varepsilon = \pm1.
   \end{equation}
If $\varepsilon=1$, we say that the $\m$-critical RW spacetime is of {\it positive type}, while, if $\varepsilon=-1$,
we say that it is of {\it negative type}.
We call $\vartheta>0$ the {\em angular parameter}. The invariant $g_3$ and the parameter
$\mathfrak{a}$ can be written as
\begin{equation}\label{inv2}
  g_3=-\frac{\varepsilon \cosh 3\vartheta}{216},\quad
   \mathfrak{a}=\frac{2}{9\H^2}\left(1-\varepsilon \cosh 3\vartheta\right).
   \end{equation}
Let $\wp_{\vartheta,\varepsilon}$ denote the Weierstrass $\wp$-function with invariants $g_2=1/12$ and $g_3$.
Let $\omega_{\vartheta,\varepsilon}$ be its real half-period.
Then, $\wp_{\vartheta,\varepsilon}$ has the unique real form
$\widehat{\wp}_{\vartheta,\varepsilon}(\tau)=\wp_{\vartheta,\varepsilon}(\tau)$, for every $\tau\in \R$.
The function $\widehat{\wp}_{\vartheta,\varepsilon}$ is periodic, with minimal period
$2\omega_{\vartheta,\varepsilon}$,
is real analytic on the open intervals
$(2h\omega_{\vartheta,\varepsilon},2(h+1)\omega_{\vartheta,\varepsilon})$,
$h\in \Z$, and possesses double poles at the points $2h\omega_{\vartheta,\varepsilon}$, $h\in \Z$.
Moreover, $\widehat{\wp}_{\vartheta,\varepsilon}(\tau)\ge -2e_1(\vartheta,\varepsilon)$, for all $\tau\in \R$,
and $\widehat{\wp}_{\vartheta,\varepsilon}(\tau)=-2e_1(\vartheta,\varepsilon)$ if and only if $\tau = \tau_h= (2h+1)\omega_{\vartheta,\varepsilon}$, $h\in \Z$.

\subsection{Critical models of Class II and negative type}

With the notation above, let $\widehat{h}_1$, $\widehat{h}_2$ and $\widehat{h}_3$ be defined by
\begin{eqnarray*}
 \widehat{h}_1(\vartheta,\H)&=&\frac{1}{3\H^2}\left(1-2\cosh\vartheta\right)^2,\\
 \widehat{h}_2(\vartheta,\H)&=&\frac{1}{18\H^2}\left(3-2\cosh \vartheta\right)\left(1-2\cosh \vartheta\right)^2,\\
 \widehat{h}_3(\vartheta,\H)&=&-\frac{1}{6\H^2}\left(1+2\cosh \vartheta\right)\left(1-2\cosh \vartheta\right)^2.
   \end{eqnarray*}
Note that $\widehat{h}_3(\vartheta,\H)<\widehat{h}_2(\vartheta,\H)<\widehat{h}_1(\vartheta,\H)$,
for every $\H$ and for every $\vartheta >0$.

\vskip0.2cm
We can prove the following.

\begin{thm}\label{Teorema3}
Let $\mathcal M(\s)$ be
an $\m$-critical RW spacetime of Class II and negative type, with cosmological constant $\Lambda$, and
angular parameter $\vartheta>0$.
{(In this case, $\H$ is necessarily positive.)}
Let $(\RWC_{\H,\vartheta}, \ell_{\H,\vartheta})$ denote its conformal model. Then,
\[
 \Lambda \le \widehat{h}_1(\vartheta,\H),\quad \RWC_{\H,\vartheta}=\R\times \SS,
  \]
and $\ell_{\H,\vartheta}$ is the cyclic quadratic form
\begin{equation}\label{clr123}
 \ell_{\H,\vartheta}=
  \frac{9\H^2}{(12\widehat{\wp}_{\vartheta,-1}-1)^2}(-d\tau^2+g).
   \end{equation}
The form $\ell_{\H,\vartheta}$ vanishes along the totally umbilical spacelike hypersurfaces
\[
  \mathbb{S}^3_h=\{(\tau,p)\in \RWC_{\H,\vartheta} \mid \tau = 2h\omega_{\vartheta,-1}\},
   \]
$h\in \Z$, and is nondegenerate and of Lorentz type on the complement
of $\bigcup_{h\in \Z}\mathbb{S}^3_h$.
The weak energy condition is automatically satisfied. The dominant energy condition is satisfied
if and only if $\Lambda\le \widehat{h}_2(\vartheta,\H)$. The strong energy condition is satisfied if
and only if
$\widehat{h}_3(\vartheta,\H)\leq \Lambda \leq\widehat{h}_2(\vartheta,\H)$.
\end{thm}

\begin{proof}
The conformal scale function $\S$ is given by
\[
   \S(\tau)=\frac{3\H}{12\widehat{\wp}_{\vartheta,-1}(\tau)-1}.
    \]
Since $\widehat{\wp}_{\vartheta,-1}\ge -2e_1(\vartheta,\H)>1/12$ and $\S>0$, then $\H>0$.
By construction, $\S$ is
periodic, with minimal period $2\omega_{\vartheta,-1}$ and has zeroes of order $2$ located at
the points $2h\omega_{\vartheta,-1}$, $h\in \Z$. This implies that $\ell_{\H,\vartheta}$ is as in \eqref{clr12}.
Then, $\ell_{\H,\vartheta}$ is smooth on $\RWC_{\H,\vartheta}=\R\times \SS$, is nondegenerate and of
Lorentz type on $(2h\omega_{\vartheta,-1},2(h+1)\omega_{\vartheta,-1})\times \SS$, $h\in \Z$,
and vanishes at $2h\omega_{\vartheta,-1}$. The functions $\cm$ and $\cp$ are as in \eqref{ccpm}
and $\mathfrak{a}$ is as in \eqref{inv2}. Thus, $\cm$ and $\cp$ are periodic function of period
$2\omega_{\vartheta,-1}$, real analytic on $(2h\omega_{\vartheta,-1},2(h+1)\omega_{\vartheta,-1})$,
and with poles of finite order at $2h\omega_{\vartheta,-1}$. The function $\cm$ achieves its minimum
at $\tau_h=(2h+1)\omega_{\vartheta,-1}$, $h\in \Z$.
Taking into account that
\[
  \widehat{\wp}_{\vartheta,-1}(\omega_{\vartheta,-1})= -2e_1(\vartheta)=\frac{1}{6}\cosh \vartheta
  \]
we have
\[
  \min(\widetilde{\m})=\frac{1}{24\pi \H^2}(-3\H^2\Lambda+(1-2\cosh\vartheta)^2).
   \]
Thus, $\cm$ is nonnegative if and only if $\Lambda \le \widehat{h}_1(\vartheta,\H)$.
The functions $\cm+\cp$, $\cm-\cp$, and $\cm/3+\cp$ are periodic, with period $2\omega_{\vartheta,-1}$,
and attain their minima at the points $\tau_h=(2h+1)\omega_{\vartheta,-1}$, $h\in \Z$. Then
\begin{eqnarray*}
 \min(\widetilde{\m}+\cp)&=&\frac{1}{72\pi \H^2}\left(1-2\cosh\vartheta\right)^2\left(3+2\cosh \vartheta\right)>0,\\
  \min(\widetilde{\m}-\cp)&=&-\frac{\Lambda}{4\pi}     +\frac{1}{72\pi   \H^2}\left(1-2\cosh\vartheta\right)^2\left(3-2\cosh\vartheta\right),\\
\min(\widetilde{\m}/3+\cp)&=&\frac{\Lambda}{12\pi}+\frac{1}{72\pi \H^2}\left(1-2\cosh \vartheta\right)^2\left(1+2\cosh \vartheta\right).
    \end{eqnarray*}
These formulae imply that the weak energy condition is automatically satisfied, that the dominant energy condition
is satisfied if and only if $\Lambda\le \widehat{h}_2(\vartheta,\H)$, and that the strong energy condition
is satisfied if and only if $\widehat{h}_3(\vartheta,\H)\le \Lambda\le \widehat{h}_2(\vartheta,\H)$.
\end{proof}

\subsection{Critical models of Class II and positive type}
Retaining the notation introduced at the beginning of the section, we now consider the $\m$-critical RW
universes of Class II and positive type. In this case, the real root is
\[
   -2e_1(\vartheta)=-\frac{1}{6}\cosh \vartheta < \frac{1}{12}.
     \]
Thus, there exists a unique $\tau^*\in (0,\omega_{\vartheta,1})$, such that $\widehat{\wp}_{\vartheta,1}(\tau^*)=1/12$.
The zeroes of the equation $\widehat{\wp}_{\vartheta,1}=1/12$ are
\[
  \tau^*_h=\tau^*+2h\omega_{\vartheta,1},\quad \widehat{\tau}^*_h=-\tau^*+2h\omega_{\vartheta,1},\quad h\in \Z.
  \]


If $\H >0$, the conformal scale function $\S$ is positive on the
intervals $ \mathrm{I}'_h=(\widehat{\tau}^*_{h},\tau^*_h)$,
$h\in \mathbb Z$, and is negative on the intervals $\mathrm{I}''_h=(\tau^*_h, \widehat{\tau}^*_{h+1})$,
 $h\in \mathbb Z$.
On the other hand, if $\H <0$, the conformal scale function $\S$ is positive on the
intervals $\mathrm{I}''_h=(\tau^*_h, \widehat{\tau}^*_{h+1})$,
$h\in \mathbb Z$, and is negative on the intervals $ \mathrm{I}'_h=(\widehat{\tau}^*_{h},\tau^*_h)$,
$h\in \mathbb Z$.

Taking into account that $\S$ is by definition a nonnegative function and
that $\widehat{\wp}_{\vartheta,1}$ is periodic with period $2\omega_{\vartheta,1}$,
we may consider $\mathrm{I}':=\mathrm{I}'_0$ or $\mathrm{I}'':=\mathrm{I}''_0$ as the maximal intervals
of definition for $\S$, depending on whether $\H$ is positive or negative.

\vskip0.2cm
For the case in which $\H$ is positive, we have the following.

\begin{thm}\label{Teorema4}
Let $\mathcal M(\s)$ be
an $\m$-critical RW space time of Class II and positive type, with cosmological constant $\Lambda$,
angular parameter $\vartheta>0$, and positive internal parameter $\H$.
Let $(\RWC_{\H,\vartheta}, \ell_{\H,\vartheta})$
denote its conformal model. Then,
\[
  \Lambda \le \frac{2\cosh 3\vartheta -2}{9\H^2},\quad \RWC_{\H,\vartheta}=\mathrm{I}'\times \SS,
   \]
and $\ell_{\H,\vartheta}$ is the non-cyclic quadratic form
\begin{equation}\label{ecf}
\ell_{\H,\vartheta}=
 \frac{9\H^2}{(12\widehat{\wp}_{\vartheta,1}-1)^2}(-d\tau^2+g).
  \end{equation}
The form $\ell_{\H,\vartheta}$ vanishes along the totally umbilical spacelike hypersurface
$$
   \mathbb{S}^3_0=\left\{(\tau,p)\in \RWC_{\H,\vartheta} \mid \tau = 0\right\}
     $$
and is nondegenerate of Lorentzian signature on the complement of $\mathbb{S}^3_0$.
The weak and the dominant energy conditions are automatically satisfied.
The strong energy condition is satisfied if and only if $\Lambda = 2(\cosh 3\vartheta -1)/9\H^2$.
\end{thm}

\begin{proof}
Let $\mathrm I' =\mathrm{I}'_{0}$ be the maximal domain of definition of $\S$.
Since $\H>0$, we have that $\widehat{\wp}_{\vartheta,1}(\tau)>1/12$, for each $\tau\in \mathrm{I}'$.
%
%
This implies that $\RWC_{\H,\vartheta}=\mathrm{I}'\times \SS$ and that the quadratic form
$\ell_{\H,\vartheta}$ is as in \eqref{ecf}.
The conformal factor $9\H^2/(12\widehat{\wp}_{\vartheta,1}-1)^2$ vanishes at $\tau=0$,
is strictly positive
for each $\tau\in \mathrm I'$, $\tau\neq 0$, and tends to $+\infty$ when $\tau \mapsto \pm \tau^*$.
The energy density $\cm$ and the pressure $\cp$ are as in \eqref{ccpm},
with $\mathfrak{a}=\frac{2}{9\H^2}(1-\cosh 3\vartheta)$.
 The density $\cm$ has a pole at the origin  and attains its minimum at $\pm \tau^*$.
 Taking into account that
\[
  \widehat{\wp}_{\vartheta,1}(\pm \tau^*)=\frac{1}{12},
  \quad \widehat{\wp}'_{\vartheta,1}(\pm \tau^*)=
   \mp \frac{1}{6\sqrt{6}}\sqrt{\cosh 3\vartheta -1},
    \quad \widehat{\wp}''_{\vartheta,1}(\pm \tau^*)=-\frac{1}{8},
     \]
we obtain
\[
  \min(\cm)=\frac{1}{72\pi \H^2}\left(2\cosh 3\vartheta - 9\H^2\Lambda-2\right).
   \]
Therefore, $\cm\ge 0$ if and only if
\[
   \Lambda \le \frac{2\cosh 3\vartheta - 2}{9\H^2}.
    \]
Similarly, $\cm+\cp$, $\cm-\cp$, and $\cm/3+\cp$ attain their minima at $\pm \tau^*$.
In particular, $\min(\cp) = -\min(\widetilde{\m})$.
Consequently, we have
\begin{eqnarray*}
&&\min(\widetilde{\m}+\cp)=0,\\
&&\min(\widetilde{\m}-\cp)=\frac{1}{36\pi \H^2}\left(2\cosh 3\vartheta -9\H^2\Lambda-2\right),\\
&&\min(\widetilde{\m}/3+\cp)=-\frac{1}{108\pi \H^2}\left(2\cosh 3\vartheta -9\H^2\Lambda-2\right).
   \end{eqnarray*}
This implies that the weak and the dominant energy conditions are automatically satisfied, while the strong energy
 condition is enforced if and only if $\Lambda = 2(\cosh(3\vartheta)-1)/9\H^2$. This proves the result.
 \end{proof}


\vskip0.2cm
As for the case in which $\H$ is negative, we can prove the following.

\begin{thm}\label{Teorema5}
Let $\mathcal M(\s)$ be
an $\m$-critical RW space time of Class II and positive type, with cosmological constant $\Lambda$,
angular parameter $\vartheta>0$, and negative internal parameter $\H$.
Let $(\RWC_{\H,\vartheta}, \ell_{\H,\vartheta})$ denote
its conformal model. Then,
\[
  \Lambda \le \min\left(\frac{2\cosh 3\vartheta -2}{9\H^2},
 \frac{\left(1+2\cosh \vartheta\right)^2}{3\H^2} \right),
   \quad \RWC_{\H,\vartheta}=\mathrm{I}''\times \SS,
   \]
and $\ell_{\H,\vartheta}$ is the non-cyclic Lorentz metric
\begin{equation}\label{ecf5}
 \ell_{\H,\vartheta}=
  \frac{9\H^2}{(12\widehat{\wp}_{\vartheta,1}-1)^2}(-d\tau^2+g).
    \end{equation}
The weak and the dominant energy conditions are satisfied if and only if
$0<\vartheta\leq\mathrm{arcosh}(3/2)$.
The strong energy condition is satisfied if and only if
$0<\vartheta\leq\mathrm{arcosh}(3/2)$
and $\Lambda=2(\cosh 3\vartheta - 1)/9\H^2$.
\end{thm}

\begin{proof}
Let $\mathrm{I}'' = \mathrm I_0'' \subset (0, 2\omega_{\vartheta,1})$ be the maximal interval of definition of $\S$.
Since $\H<0$, then
$\widehat{\wp}_{\vartheta,1}(\tau)<1/12$, for every $\tau\in \mathrm{I}''$.
The conformal energy density $\cm$ is real analytic on
$(0,2\omega_{\vartheta,1})$.
If $\vartheta \in (0,\mathrm{arcosh}(3/2))$, then $\cm$ has
 relative minima at $\tau=\tau^*,\widehat{\tau}^*$,
and a relative maximum at $\omega_{\vartheta,1}$. We then have
\begin{equation}\label{ee}
 \min_{\overline{\mathrm{I}''}} (\cm)=\cm(\tau^*)=\cm(\widehat{\tau}^*)
  =\frac{2(\cosh 3\vartheta -1)-9\H^2\Lambda}{72\H^2\pi}.
   \end{equation}
Therefore, if $\vartheta\in (0,\mathrm{arcosh}(3/2))$, $\cm\ge 0$
if and only if $\Lambda\le 2(\cosh 3\vartheta - 1)/9\H^2$.
If $\vartheta >\mathrm{arcosh}(3/2)$, then $\cm$ has three relative minima
on
$(0,2\omega_{\vartheta,1})$,
located at $\tau^*$, $\widehat{\tau}^*$, and $\omega_{\vartheta,1}$.
The value of $\cm$ at $\tau^*,\widehat{\tau}^*$ is as in \eqref{ee}, while
\[
  \cm(\omega_{\vartheta,1})=\frac{\left(1+2\cosh \vartheta\right)^2-3\H^2\Lambda}{24\H^2\pi}.
  \]
Consequently, if $\vartheta\in (\mathrm{arcosh}(3/2),\mathrm{arcosh}(5/2))$,
$\tau^*$ and $\widehat{\tau}^*$ are minimum points, and
%
\[
  \min_{\overline{\mathrm{I}''}} (\cm)=\frac{2(\cosh 3\vartheta -1)-9\H^2\Lambda}{72\H^2\pi}.
   \]
If  $\vartheta =\mathrm{arcosh}(5/2)$, the values of $\cm$ on $\tau^*$, $\widehat{\tau}^*$, and
$\omega_{\vartheta,1}$ do coincide.
If $\vartheta>\mathrm{arcosh}(5/2)$,
the minimum of $\cm$ is attained at $\omega_{\vartheta,1}$, and
%
\[
 \min_{\overline{\mathrm{I}''}} (\cm)=\frac{(1+2\cosh \vartheta )^2-3\H^2\Lambda}{24\H^2\pi}.
   \]
Then, if $\vartheta\in \left(\mathrm{arcosh}(3/2),\mathrm{arcosh}(5/2)\right)$, $\cm\ge 0$ if and only if
$\Lambda\le 2(\cosh 3\vartheta -1)/9\H^2$; and
if $\vartheta> \mathrm{arcosh}(5/2)$, $\cm\ge 0$ if and only if
$\Lambda\le (1+2\cosh \vartheta)^2/3\H^2$. Since $2(\cosh 3\vartheta-1)\le 3(1+2\cosh\vartheta)^2$ on
$\left[0,\mathrm{arcosh}(5/2)\right]$ and $3(1+2\cosh \vartheta)^2\le 2(\cosh 3\vartheta -1)$
on $\left[\mathrm{arcosh}(5/2),+\infty\right)$, it follows that $\cm$ is
nonnegative on $\mathrm I''$ if and only if
$\Lambda\le
\min\left(\frac{2\cosh 3\vartheta -2}{9\H^2},
 \frac{\left(1+2\cosh \vartheta\right)^2}{3\H^2} \right)
$.

The function $\cm+\cp$ has relative minima
 at $\tau^*$, $\widehat{\tau}^*$,
 and $\omega_{\vartheta,1}$ in the interval
 $(0,2\omega_{\vartheta,1})$ . In particular, $\cm+\cp$ vanishes at
 $\tau^*$ and $\widehat{\tau}^*$, while its value at $\omega_{\vartheta,1}$ is
\[
 \frac{(3-2\cosh \vartheta)(1+2\cosh \vartheta)^2}{72\H^2\pi}.
  \]
Thus $\cm+\cp$ is nonnegative if and only if $\vartheta\in (0,\mathrm{arcosh}(3/2)]$.
This means that the weak energy condition is satisfied if and only if $0<\vartheta \le \mathrm{arcosh}(3/2)$.

For $\vartheta\in (0,\mathrm{arcosh}(3/2)]$,
the function $\cm-\cp$ has two
minima at $\tau^*$, $\widehat{\tau}^*$, and
\[
  \cm(\tau^*)-\cp(\tau^*)=\cm(\widehat{\tau}^*)-\cp(\widehat{\tau}^*)=
- \frac{9\H^2\Lambda-2(\cosh 3\vartheta -1)}{36\H^2\pi}.
  \]
This implies that the dominant energy condition is automatically satisfied.

For $\vartheta\in (0,\mathrm{arcosh}(3/2)]$, the minimum of $\cm/3+\cp$ on
the interval $(\tau^*,\widehat{\tau}^*)$ is attained at the half period $\omega_{\vartheta,1}$, and
\[
  \frac{\cm(\omega_{\vartheta,1})}{3}+\cp(\omega_{\vartheta})
  =\frac{9\H^2\Lambda-2(\cosh 3\vartheta -1)}{108\H^2\pi}.
   \]
Accordingly, the strong energy condition is satisfied if and only if $\Lambda=2(\cosh 3\vartheta -1)/9\H^2$,
as claimed.
\end{proof}

\bibliographystyle{amsalpha}

\end{document}